\documentclass[10pt]{article}
\usepackage[vmargin=2.5cm]{geometry}
\usepackage{graphicx} 
\usepackage{amssymb}
\usepackage{amsthm}
\usepackage{amsmath}
\usepackage{mathbbol}
\usepackage{setspace}
\usepackage{parskip}
\usepackage{hyperref}
\usepackage{microtype}
\allowdisplaybreaks
\hypersetup{colorlinks=true,
linkcolor=blue,
citecolor=blue,
urlcolor=cyan}
\newtheorem{theorem}{Theorem}[section]
\newtheorem{lemma}{Lemma}[section]
\newtheorem{remark}{Remark}[section]
\newtheorem{conjecture}{Conjecture}[section]
\title{An Integral Mean Value Theorem for Weyl Sums over Broken Arcs}
\author{YaoJie Guo}
\date{June 2026}
\begin{document}
\maketitle
\begin{abstract}
In this article, we research the mean value of integral of exponential sum $S(\alpha)=\sum_{u\in I}e(\alpha u^k)$, where $I$ is a short interval whose length is $2N^{\theta},\theta<1$, and the broken arc $\mathfrak{m^*}$ is a subset of a following minor arcs
\[
\mathfrak{m}=\bigcap_{j\leq k-1}\left\{\alpha:\forall q<(\log N)^A,h<q,(h,q)=1,\Big|\alpha-\frac{h}{q}\Big|>\frac{1}{qN^{(k-j-1/2)\theta}}\right\}
\]
which has measure at least $c>0$. By setting $m$ is a sufficiently large number, $N$ is be sufficiently large in terms of $m$. When $k>\log m$ and $\frac{\log k}{\log m}<1/2$ we set the following estimate:
\[
\int_{\mathfrak{m^*}}\bigg|\sum\limits_{{N_1}<u<{N_2}}{e(zu^k)}\bigg|^{m}\mathrm{d}z\ll_cN^{\theta m(\frac{2k+1}{2k+2}+o(1))}
\]
We can find this estimate moving beyond the even-odd restriction of powers. To get this bound, we first set a strong estimate for almost $\alpha$ by Diophantine approximation and Vinogradov's main value theorem. Then, combining this result, we construct a refined Weyl differencing argument by partitioning the differences step into large and small range, which significantly outperforms the classical one. By the version of probability, we can calculate the multiplicity of each sum. Put them together and we can complete the proof.
\end{abstract}
\section{The Introduction}
In this paper we will research the mean value theorem of exponential sum. Here we set $e(*)=e^{2\pi i*}$, and $\textbf{a}=(a_1,...,a_k)$.  The integral
\begin{equation}
I_1=\int_{[0,1]^k}\bigg|\sum\limits_{u\leq N}e(a_ku^k+...+a_1u)\bigg|^{s}\mathrm{d}\textbf{a}
\end{equation}
and the more direct version
\begin{equation}
I_2=\int_{[0,1]}\bigg|\sum\limits_{u\leq N}e(\alpha u^k)\bigg|^{s}\mathrm{d}\alpha
\end{equation}
can be seen as the mean value of $s-th$ moment of he integral of high-degree of Weyl sums. The best result so far are all concentrated in first case. One of the most important results is the Vinogradov's Main Conjecture:
\[
\int_{[0,1]^k}\bigg|\sum\limits_{u\leq N}e(a_ku^k+...+a_1u)\bigg|^{s}\mathrm{d}\textbf{a}\ll N^{\epsilon}(N^{s-k(k+1)/2})+N^{s/2})
\]
and there is a famous result according to \cite{wooley2015discretefourierrestrictionefficient}:
\[
\int_{[0,1]}\bigg|\sum\limits_{u\leq N}e(\alpha u^k)\bigg|^{s}\mathrm{d}\alpha\ll N^{\epsilon}(N^{s-k}+N^{s/2})
\]
Here we need $s$ is a even number. The conjecture has been proved by \cite{bourgain2016proofmainconjecturevinogradovs} for degree $k>3$, and \cite{wooley2014cubiccasemainconjecture} for degree $k=3$. Another form is seldom to research, we can see the smooth form in \cite{bruedern2024estimatessmoothweylsums} 
Base on these theorem, we can see some different directions of current research.\\
\textbf{Weighted Sums.} This direction is to research the weighted sum or its mean value, like
\[
\int_{[0,1]^k}\bigg|\sum\limits_{u\leq N}f(u)e(a_ku^k+...+a_1u)\bigg|^{s}\mathrm{d}\textbf{a}
\]
where $|f(u)|$ can be $1$(the usual situation), $0$(the indicator function, to restrict in short interval) or other multiple function. Some current research could be found in \cite{chen2020metrictheoryweylsums} for the measure and point-conclusion, and \cite{wooley2015discretefourierrestrictionefficient}, for the mean value conclusion, the application could be found in \cite{wei2014sumspowersequalprimes}.
Moreover, here we can find the generation of Vinogradov's main conjecture:
\[
\int_{[0,1]^k}\bigg|\sum\limits_{|u|\leq N}f(u)e(a_ku^k+...+a_1u)\bigg|^{2s}\mathrm{d}\textbf{a}\ll N^{\epsilon}(1+N^{s-k(k+1)/2})(\sum\limits_{|n|\leq N}f(n)^2)^s
\]
where $\epsilon>0$ is any given positive number. By the way, we have not find any similar conjecture for the second form.\\
\textbf{Broken Arcs.} This direction turns the whole arcs to broken arcs. Some usual situation is the major arcs in circle method, Goldbach-Waring problem, like \cite{wei2014sumspowersequalprimes}. Some difficult versions can be found in \cite{oh2025extendedvinogradovsmeanvalue}, which establish the mean value theorem around the arcs near to $0$ in some components. Generally speaking, the results have following form:
\[
\int_{\mathfrak{D}}\bigg|\sum\limits_{u\leq N}e(a_ku^k+...+a_1u)\bigg|^{s}\mathrm{d}\textbf{a}\ll L(\mathfrak{D},N)N^{s/2+\epsilon}
\]
or the single point version:
\[
\max_{\textbf{a}\in\mathfrak{D}}\bigg|\sum\limits_{u\leq N}e(a_ku^k+...+a_1u)\bigg|^{s}\ll L(\mathfrak{D},N)N^{s/2+\epsilon}
\]
where $L(\mathfrak{D},N)$ is rely on the structure of $\mathfrak{D}$. Current papers mainly discuss the situation that $\mathfrak{D}=[0,1)^{k-1}\times[0,N^{-\mu})$, and $L(\mathfrak{D},N)=N^{-\mu}$. For the general situation, it is still hard to give the result.\\
\textbf{The version of measure theory.}
Recently, the version of number theory appeared in the exponential sum. First it can be seen in \cite{wooley2015perturbationsweylsums}, and then we can see another example in \cite{chen2020metrictheoryweylsums}. The traditional ideal problem is: To find the bound
\[
\sum\limits_{u\leq N}e(a_ku^k+...+a_1u)\ll L(N)
\]
holds true for almost any given $\textbf{a}=(a_1,...,a_k)$. But the reality is, it seems hard to give a sharp bound for any sufficiently large $N$ without any condition add in $\alpha$, which caused the result cannot apply on the mean value theorem.\\
In this paper, we will combine the first and the second versions together. We set
\[
\mathfrak{m}=\bigcap_{j\leq k-1}\left\{\alpha:\forall q<(\log N)^A,h<q,(h,q)=1,\Big|\alpha-\frac{h}{q}\Big|>\frac{1}{qN^{(k-j-1/2)\theta}}\right\}
\]
which is the simple minor arc but not the minor arc in Waring Problem. Our main goal is not only to establish the mean value theorem for $\mathfrak{m}$, but also for any subset of $\mathfrak{m}$(We call it "Broken arc" in this paper). Then, we focus on the short interval, whose length is $CN^{\theta},\theta<1$. More exactly, we have,
\begin{theorem}
For the sufficiently large number $m$, $m^{1/2-\delta-\epsilon}>k>\log m$ and any arcs $\mathfrak{m^*}\subset\mathfrak{m}$ whose Lebesgue measure $\mathfrak{L}(\mathfrak{m^*})>c$, and $c$ is a constants not rely on $m,N$, we have
\[
\int_{\mathfrak{m^*}}\bigg|\sum\limits_{{N_1}<u<{N_2}}{e(zu^k)}\bigg|^{m}\mathrm{d}z\ll_cN^{\theta m(\frac{2k+1}{2k+2}+m^{-1/2-\delta})}
\]
where $\epsilon$ are any given positive number, $\theta<1$ and $\ll$ are a positive constants not rely on $m,\mathfrak{m}$.
\end{theorem}
and after section 2, we can give another theorem:
\begin{theorem}
For the sufficiently large number $m,k$ and any arcs $\mathfrak{m^*}\subset\mathfrak{m}$ whose Lebesgue measure $\mathfrak{L}(\mathfrak{m^*})>c$, and $c$ is a constants not rely on $m,N$, we have
\[
\int_{\mathfrak{m^*}}\bigg|\sum\limits_{{N-N^\theta}<u<{N+N^\theta}}{e(zu^k)}\bigg|^{m}\mathrm{d}z\ll N^{\theta(m-\frac{m}{4(k+1)})+\epsilon'}
\]
where $\epsilon'$ is any given positive number. $\theta<1$, and $C$ is a positive constants not rely on $m,\mathfrak{m}$ and $\mathfrak{m^*}$.
\end{theorem}
To prove the Theorem 1.2, we will prove lemma 2.1, lemma 2.2 at first. By the way, due to the version of measure theory and Algebraic Geometry, we can overcomes the parity of the moments in substance (The meaning is to prove the general theorem without norm inequality), which is hard to do in the past.\\
Our main idea is a method called "The Efficient Partition of Weyl Difference". After doing simple Weyl difference, we partition the sum into three parts: The sum of large difference (case I); The sum of small difference and good Diophantine quality (case II); The sum of small difference and bad Diophantine quality (case III). We will prove that the contribution of case III is far less than case II totally, so we just need to calculate the case I and case II. Then, when the exponential sum of case I is larger than that of case II, we keep doing the Weyl Difference; When the exponential sum of case I is smaller than that of case II, we stop the doing the difference and use the bound of the sum of case II to calculate the exponential sum. From setting the lemma 4.1 and 4.2, we can get the exact proportion between case I and case II, and we can calculate the contributions of these cases. Therefore, we will hold a parameter optimization problem (with self-reference phenomenon) in Section 6, then we can prove the theorem 1.1.\\
The most important tool during the process of calculating the contribution of case II is lemma 2.2. First we use the classical method in \cite{PP1986} to transform the Weyl sum into a several product of linear exponential sum and the integral in Vinogradov's main conjecture. According to the result of main conjecture, we can bound the second part. Then, we can bound the first part through the Diophantine Approximation of the count in broken minor arcs.\\
After showing the outline of this paper, we still have some conjecture wait to solve in the future.
\begin{conjecture}
Let $f(u)=\textbf{a}\cdot(u,u^2,...u^k)$. When $m,k$ satisfied $m^{1/2-\delta-\epsilon}>k>\log m$ are sufficiently large, and $N$ is sufficiently large in terms of $m,k$, there exist a arc $\mathfrak{B}\subset[0,1)^k$, and $|\mathfrak{B}|=1+o(1)$, hold following inequality:
\[
\int_\mathfrak{B'}\Big|\sum\limits_{N-N^\theta\leq u< N+N^\theta}e(f(u))\Big|^m\mathrm{d}\mathbf{a}\ll |\mathfrak{B'}|N^{\theta m(1-\frac{1}{2(k+1)}+m^{-1/2-\delta})}
\]
where $\theta<1/2$, $\epsilon,\epsilon'$ are any given positive number, and $\mathfrak{B'}\subset\mathfrak{B}$ is any given subset.
\end{conjecture}
This question maybe could be solved by the technique like lemma 2.2, and we need to develop the Diophantine simultaneous approximation skill.
\begin{conjecture}
For the sufficiently large numbers $m^{1/2-\delta-\epsilon}>k>\log m$ and any arcs $\mathfrak{m^*}\subset\mathfrak{m}$ whose Lebesgue measure $\mathfrak{L}(\mathfrak{m^*})>c$, and $c$ is a constants not rely on $m,N$, we have
\[
\int_{\mathfrak{m^*}}\bigg|\sum\limits_{{N-N^\theta}<u<{N+N^\theta}}{e(zu^k)}\bigg|^{m}\mathrm{d}z
\ll|\mathfrak{m^*}|N^{\theta m(1-\frac{1}{2(k+1)})+o(1)}
\]
where $\epsilon,\epsilon'$ are any given positive number, $\theta\leq1$
\end{conjecture}
This conjecture is the similar to what we do, and we believe the result can be proved in nearly future. After all, we still have a conjecture during the research, which is the improvement of lemma 2.2.
\begin{conjecture}
Let $f(x)\in\textbf{Z}[x],\deg f=k$ is a monic polynomial, and $H=[N-N^\theta,N+N^\theta]$. The following estimate
\[
\sum\limits_{\substack{x\in H}}e(\alpha f(x))\ll_\epsilon N^{\theta(1-\frac{1}{k})+\epsilon}
\]
hold true for almost any given $\alpha\in\mathfrak{m^*}$.
\end{conjecture}
We guess the bound of conjecture 1.3 is the best, for it is the classical result of Hua,L.K.,\cite{Hua1963},\cite{PP1986} and \cite{Weil1948}, and it could not be improve in a large range. In the process of completing this paper, we try to prove this inequality for several times, but it seems hard to complete it. The tool we used before is Duffin-Schaeffer theorem, which comes from \cite{koukoulopoulos2024sharpquantitativeversionduffinschaeffer}, and it is used to transform the Weyl sum into rational sum.
\section{The pointwise bound according to Vinogradov's main value theorem}
In this section we will prove the following lemmas:(By the way the definition of letter $m$ in this section in not the same as it in the next section.)
\begin{lemma}
We set $f(x)=\sum_{1\leq i\leq k+1}\alpha c_{i}x^i$, where $l,N,X,Y$ are four integers, where $Y<X$ is sufficiently large and $X=o(N)$, $l>k(k+1)/2$. In this condition, for any given $\epsilon>0$, we have
\begin{equation}
|\sum\limits_{n=N}^{N+X}e(f(n))|\ll Y+X^{1-\frac{kc}{4l}+\epsilon}
\end{equation}
hold true for almost any given $\alpha$ belong to
\[
\mathfrak{m}(Y,X)=\bigcap_{j\leq k-1}\left\{\alpha:\forall q<Y^j,h<q,(h,q)=1,\Big|\alpha-\frac{h}{q}\Big|>\frac{1}{X^{k-j-c}q}\right\}
\]
and $\ll$ is a absolute constant.
\end{lemma}
\begin{proof}
We set $g(n)=f(n+N)$, and we know
\[
|\sum\limits_{n=N}^{N+X}e(f(n))|=|\sum\limits_{n=0}^{X}e(g(n))|
\]
and the first coefficient of $g(n)$ is the same as $f(n)$, so we just need to prove
\[
|\sum\limits_{n=0}^{X}e(f(n))|\ll Y+X^{1-\frac{kc}{4l}+\epsilon}
\]
From \cite{PP1986}, we know 
\[
|\sum\limits_{n=0}^{X}e(f(n))|=Y^{-1}\sum\limits_{y\leq Y}\sum\limits_{n\leq X}e(f(x+y))+O(Y)
\]
and we have
\[
Y^{-1}\Big|\sum\limits_{y\leq Y}\sum\limits_{n\leq X}e(f(x+y))\Big|^{2l}\leq Y^{-1}X^lI^{1/2}(X;k,l)\Big(\sum\limits_{\lambda_1,...\lambda_k}\big|\sum\limits_{y\leq Y}e(f^{(1)}(y)\lambda_1+...+\frac{1}{k!}f^{(k)}(y)\lambda_k)\big|^2\Big)^{1/2}
\]
where $\lambda_j<lX^j$ and
\[
I(X;k,l)=\int_{[0,1]^k}\bigg|\sum\limits_{u\leq X}e(a_ku^k+...+a_1u)\bigg|^{2l}\mathrm{d}\textbf{a}
\]
By Vinogradov's main conjecture(To see in \cite{bourgain2016proofmainconjecturevinogradovs} for $k>3$, and \cite{wooley2014cubiccasemainconjecture} in condition that $k=3$) we know
\[
I(X;k,l)\ll X^{\epsilon}(X^{l}+X^{2l-k(k+1)/2})
\]
From Holder inequality(The detail can be found in \cite{PP1986}, just some elementary skills), we have
\[
Y^{-2l}\Big|\sum\limits_{y\leq Y}\sum\limits_{n\leq X}e(f(x+y))\Big|^{2l}\leq Y^{-1}X^{2l-k(k+1)/4+\epsilon}\Big(\sum\limits_{\lambda_1,...\lambda_k}\big|\sum\limits_{y\leq Y}e(f^{(1)}(y)\lambda_1+...+\frac{1}{k!}f^{(k)}(y)\lambda_k)\big|^2\Big)^{1/2}
\]
Then we know
\begin{align*}
&\sum\limits_{\lambda_1,...\lambda_k}\big|\sum\limits_{y\leq Y}e(f^{(1)}(y)\lambda_1+...+\frac{1}{k}f^{(k)}(y)\lambda_k)\big|^2\\
=&\sum\limits_{y\leq Y}\sum\limits_{y'\leq Y}\sum\limits_{\lambda_1,...\lambda_k}e((f^{(1)}(y)-f^{(1)}(y'))\lambda_1+...+\frac{1}{k!}(f^{(k)}(y)-f^{(k)}(y'))\lambda_k)\\
=&\sum\limits_{y\leq Y}\sum\limits_{y'\leq Y}\prod\limits_{1\leq j\leq k}\sum\limits_{\lambda_j<lX^j}e\big(\frac{1}{j!}(f^{(j)}(y)-f^{(j)}(y'))\big)
\end{align*}
We know
\[
\sum\limits_{\lambda_j<kX^j}e\big(\frac{1}{j!}(f^{(j)}(y)-f^{(j)}(y'))\lambda_j\big)\ll_k\min\bigg(X^j,\frac{1}{2\|\frac{1}{j!}(f^{(j)}(y)-f^{(j)}(y'))\|}\bigg)
\]
From the definition of $\mathfrak{m}(Y,P)$, we know
\[
\|\frac{1}{j!}(f^{(j)}(y)-f^{(j)}(y'))\|\gg X^{j}
\]
Then we know
\[
\prod\limits_{1\leq j\leq k}\sum\limits_{\lambda_j<lX^j}e\big(\frac{1}{j!}(f^{(j)}(y)-f^{(j)}(y'))\big)\ll\prod\limits_{1\leq j\leq k}\frac{1}{\|\frac{1}{j!}(f^{(j)}(y)-f^{(j)}(y'))\|}
\]
Then we have
\[
\prod\limits_{1\leq j\leq k}\sum\limits_{\lambda_j<lX^j}e\big(\frac{1}{j!}(f^{(j)}(y)-f^{(j)}(y'))\big)\ll X^{(k+1)k/2-kc}
\]
and
\[
\sum\limits_{\lambda_1,...\lambda_k}\big|\sum\limits_{y\leq Y}e(f^{(1)}(y)\lambda_1+...+\frac{1}{k!}f^{(k)}(y)\lambda_k)\big|^2\ll Y^2X^{k(k+1)/2-kc}
\]
Hence we have
\[
Y^{-1}\sum\limits_{y\leq Y}\sum\limits_{n\leq X}e(f(x+y))\ll X^{2l-\frac{kc}{2}+\epsilon}
\]
So we complete the proof.
\end{proof}
To ensure our estimate is effective for nearly almost any given $\alpha$, we need
\[
\mathfrak{L}([0,1]-\mathfrak{m}(Y,X))=o(1)
\]
which means
\[
\frac{Y^j}{X^{k-j-c}}=o(1),\forall j>1
\]
choosing $Y=(\log x)^A$, $c=1/2+\epsilon$ and we have following theorem:
\begin{lemma}
Let $f(x)\in\textbf{Z}[x],\deg f=k$ is a monic polynomial, and $H=[N-N^\theta,N+N^\theta]$. The following estimate
\[
\sum\limits_{\substack{x\in H}}e(\alpha f(x))\ll_\epsilon N^{\theta(1-\frac{1}{4(k+1)})+\epsilon}
\]
hold true for any given $\epsilon>0$ and any $\alpha\in\mathfrak{m}$.
\end{lemma}
\section{The example of Efficient Partition of Weyl Difference}
We set $U=2^ny,(2,y)=1,z_1=a/T+a^*/T^2,a^*\in(0,1)$, and
\[
z_j(h_1,..,h_j)=\frac{k!z_1}{(k-j)!}\prod\limits_{i\leq j}h_i
\]
We definite
\begin{align*}
f_{j+1}(u_1,h_1,..,h_{j+1})=&f_{j}(u_1+h_{j+1},h_1,..,h_{j})-f_{j}(u_1,h_1,..,h_{j})\\+&z_j(h_1,..,h_j)((u_1+h_{j+1})^{k-j}-u_1^{k-j})-z_{j+1}(h_1,..,h_j,h_{j+1})
\end{align*}
where $f_0=1$ and $h_j$ will appear in our following process. Due to our problem is in short interval, The number of solutions is not too big. This character inspire us to use Weyl Differential more accuracy in following way. More exactly, We have
\begin{lemma}
Let $T(K)$ to be the number of solutions of the equation 
\[
|f(x)-f(y)|\leq K\quad x,y\in[N-N^\theta,N+N^\theta]
\]
where the first coefficient of $f$ is one. Setting $K=o(N^{k})$ and $N^{k-1}=o(K)$ we have
\[
T(K)=\frac{2K}{kN^{k-1-\theta}}(1+O(\frac{1}{N}))+O(\frac{K^{2}}{N^{2k-2-\theta}})+O(N^\theta)
\]
where $\epsilon$ is any given positive number.
\end{lemma}
Then we need to prove that the Diophantine property of $z_i$ is the same. 
\begin{lemma}
For any given $\alpha$ we have
\[
z_j(h_1,..,h_j)=\frac{k!z_1}{(k-j)!}\prod\limits_{i\leq j}h_i\notin\mathfrak{m}
\]
happen with the probability less than $2N^{-1/2+\epsilon}$.
\end{lemma}
\begin{proof}
We set $G_\alpha(x)=x+\alpha$, and the whole orbit $\left\{T^n_\alpha(x)\right\}$ is a stochastic process. We know
\begin{align*}
\sum\limits_{q<(\log N)^A}P(q\alpha\notin\mathfrak{m})
=&\sum\limits_{j<N^{k/2}}P(||T^j_\alpha(\alpha)||\notin\mathfrak{m})\\
\leq&\sum\limits_{j<(\log N)^A}P(||T^j_\alpha(\alpha)||\leq N^{-1/2})\\
=2&\sum\limits_{j<(\log N)^A}\mathfrak{L}(T^{-n}_\alpha((0,N^{-1/2}]))<\frac{2}{N^{1/2-\epsilon}}
\end{align*}
Thus the probability to have a solution is less than $2N^{-1/2+\epsilon}$ for any $\alpha$. 
\end{proof}
To make our expression clear, we will ignore the little positive number $\epsilon>0$. The prove of lemma 3.1 can be found in Appendix. Noting that
\[
\text{Bad}(j)=\left\{h_j\Big|h_{j-1}\in\text{Bad}(j-1),z_j(h_1,..,h_j)\notin\mathfrak{m}\right\}
\]
So that $|\text{Bad}(j)|\asymp N^{-j/2}$ Returning to the exponential sum, we first do original Weyl differencing for $A$ times. For these steps, we have
\begin{align*}
|S^*(z_1;n,k)|\ll&\sum\limits_{N^\theta>|h_1|\geq\frac{K_1}{kN^{k-1}}}...\sum\limits_{N^\theta>|h_A|\geq\frac{K_A}{(k-A+1)N^{k-A}}}|S^*(\frac{k!z_1}{(k-A)!}\prod\limits_{i\leq A}h_i;n-A,k-A)|\\
&\prod\limits_{j\leq A}\bigg(2N^{\theta}-\frac{K_j}{(k-j+1)N^{k-j}}\bigg)^{2^{n-j}y-1}\\
=&\sum\limits_{N^\theta>|h_1|\geq\frac{K_1}{kN^{k-1}}}...\sum\limits_{N^\theta>|h_A|\geq\frac{K_A}{(k-A+1)N^{k-A}}}|\sum\limits_{{N_1}<u<{N_2}}{e(\frac{k!z_1}{(k-n)!}\prod\limits_{i\leq A}h_iu^{k-A}+f_A({u_1,h_1,...,h_A})})|^{2^{n-A}y}\\
&\prod\limits_{j\leq A}\bigg(2N^{\theta}-\frac{K_j}{(k-j+1)N^{k-j}}\bigg)^{2^{n-j}y-1}
\end{align*}
where
\[
S^*(z_1;n,k)=|\sum\limits_{{N_1}<u<{N_2}}{e(z_1u^k)+f_1({u_1,h_1})}|^{2^ny}
\]
and $K_i(i\leq A)=0$. This is the standard argument, we will not give all the detail. The only thing we need to tell is the range of sum, and it will be explained by the following example: We know that the range of the index of summation is various, but we use a unchanging number $\frac{K_1}{kN^{k-1}}$ to replace it. in this meaning, the sum will expand in the arc $B(u_1,\frac{K_1}{kN^{k-1}})$. The original range of the index is almost included in the arc above, and the number of outliers is far less than $\frac{K_1}{kN^{k-1}}$. More precisely, we have
\[
K>u_1^k-u_2^k=(u_1-u_2)\sum\limits_{1\leq j\leq k}u_1^ju_2^{k-j}=(u_1-u_2)(kN^{k-1}+O(N^{k-2+\theta}))
\]
Thus,
\[
u_1-u_2<\frac{K_1}{kN^{k-1}}+O(\frac{K_1}{kN^{k-\theta}})
\]
so the error term cannot destroy our result. Anyway, the range of the sum in every step is no more than $N^\theta>|h_i|$, so we will modify the form of range and ignore the same upper bound of range from now on.  After using the differences for $A$ step, we will use the improvement of Weyl differences. What we need to calculate is
\[
|\sum\limits_{{N_1}<u<{N_2}}{e(\frac{k!z_1}{(k-n)!}\prod\limits_{i\leq A}h_iu^{k-A}+f_A({u_1,h_1,...,h_A})}|^{2^{n-A}y}
\]
For this sum, we take one of the terms with an exponent of $2$, and square the base: 
\begin{align*}
&|\sum\limits_{{N_1}<u<{N_2}}{e(\frac{k!z_1}{(k-A)!}\prod\limits_{i\leq A}h_iu^{k-A}+f_A({u_1,h_1,...,h_A})})|^{2^{n-A}y}\\
=&\Big|\sum\limits_{{N_1}<u<{N_2}}{e(\frac{k!z_1}{(k-A)!}\prod\limits_{i\leq A}h_iu_1^{k-A}+f_A({u_1,h_1,...,h_A})-\frac{k!z_1}{(k-A)!}\prod\limits_{i\leq A}h_iu_2^{k-A}+f_A({u_2,h_1,...,h_A})})|^{2^{n-A-1}y}\\
=&\Big|\sum\limits_{\substack{|u_1^k-u_2^k|<K_1}}{e(\frac{k!z_1}{(k-A)!}\prod\limits_{i\leq A}h_iu_1^{k-A}+f_A({u_1,h_1,...,h_A})-\frac{k!z_1}{(k-A)!}\prod\limits_{i\leq A}h_iu_2^{k-A}+f_A({u_2,h_1,...,h_A}))}\\
+&\sum\limits_{\substack{|u_1^k-u_2^k|\geq K_1}}{e(\frac{k!z_1}{(k-A)!}\prod\limits_{i\leq A}h_iu_1^{k-A}+f_A({u_1,h_1,...,h_A})-\frac{k!z_1}{(k-A)!}\prod\limits_{i\leq A}h_iu_2^{k-A}+f_A({u_2,h_1,...,h_A})})\Big|^{2^{n-A-1}y}
\end{align*}
From lemma 2.2 we know
\begin{align*}
&\Big|\sum\limits_{\substack{|u_1^k-u_2^k|<K_1}}\sum\limits_{N_1<u<N_2}{e(\frac{k!z_1}{(k-A)!}\prod\limits_{i\leq A}h_iu_1^{k-A}+f_A({u_1,h_1,...,h_A})-\frac{k!z_1}{(k-A)!}\prod\limits_{i\leq A}h_iu_2^{k-A}+f_A({u_2,h_1,...,h_A})})\bigg|\\
\leq&\bigg|\sum\limits_{h_{A+1}<\frac{K_{A+1}}{(k-A)N^{k-A-1}}}\sum\limits_{N_1<u<N_2}e(\frac{k!z_1}{(k-A-1)!}\prod\limits_{i\leq A+1}h_iu_1^{k-A}+f_{A+1}({u_1,h_1,...,h_A,h_{A+1}}))\bigg|\\
\ll&\sum\limits_{h_{A+1}<\frac{K_{A+1}}{(k-A)N^{k-A-1}},h_A\in\text{Bad}(A)}\bigg|\sum\limits_{N_1<u<N_2}e(\frac{k!z_1}{(k-A-1)!}\prod\limits_{i\leq A+1}h_iu_1^{k-A}+f_{A+1}({u_1,h_1,...,h_A,h_{A+1}}))\bigg|\\
&+\frac{K_1}{kN^{k-A-1+\frac{\theta}{4(k-A+1)}}}
\end{align*}
Here we loss the exponent for $1-\theta$ to make the main term clear, and it would not influent our final conclusion.
\begin{align*}
&|\sum\limits_{{N_1}<u<{N_2}}{e(\frac{k!z_1}{(k-A)!}\prod\limits_{i\leq A}h_iu^{k-A}+f_A({u_1,h_1,...,h_A})})|^{2^{n-A}y}\\
\ll&\bigg(\bigg|\sum\limits_{h_{A+1}>\frac{K_{A+1}}{(k-A)N^{k-A-1}}}\sum\limits_{N_1<u<N_2}e(\frac{k!z_1}{(k-A-1)!}\prod\limits_{i\leq A+1}h_iu_1^{k-A}+f_{A+1}({u_1,h_1,...,h_A,h_{A+1}}))\bigg|\\
+&\frac{K_1}{kN^{k-A-1+\frac{\theta}{4(k-A+1)}}}+\bigg|\sum\limits_{h_{A+1}<\frac{K_{A+1}}{(k-A)N^{k-A-1}},h_A\in\text{Bad}(A)}e(\frac{k!z_1}{(k-A-1)!}\prod\limits_{i\leq A+1}h_iu_1^{k-A}+f_{A+1}({u_1,h_1,...,h_A,h_{A+1}}))\bigg|
\end{align*}
For another part, we know
\begin{align*}
&\bigg|\sum\limits_{N_1<u<N_2}\sum\limits_{h_{A+1}>\frac{K_{A+1}}{(k-A)N^{k-A-1}}}e(\frac{k!z_1}{(k-A-2)!}\prod\limits_{i\leq A+1}h_iu_1^{k-A}+f_{A+1}({u_1,h_1,...,h_A,h_{A+1}}))\bigg|\\
\leq&\sum\limits_{h_{A+1}<\frac{K_{A+1}}{(k-A)N^{k-A-1}}}\bigg|\sum\limits_{N_1<u<N_2}e(\frac{k!z_1}{(k-A-2)!}\prod\limits_{i\leq A+1}h_iu_1^{k-A}+f_{A+1}({u_1,h_1,...,h_A,h_{A+1}}))\bigg|
\end{align*}
Thus we have
\begin{align*}
|\sum\limits_{{N_1}<u<{N_2}}{e(\frac{k!z_1}{(k-A)!}\prod\limits_{i\leq A}h_iu^{k-A}+f_A({u_1,h_1,...,h_A}))})|^{2^{n-A}y}
\leq\max\left\{G_1,G_2\right\}
\end{align*}
where
\begin{align*}
G_1=&\bigg(\frac{K_1}{kN^{k-A-1+\frac{\theta}{4(k-A+1)}}}\bigg)^{2^{n-A-1}y}\\
G_2=&\bigg(\sum\limits_{h_{A+1}>\frac{K_{A+1}}{(k-A)N^{k-A-1}}}\bigg|\sum\limits_{N_1<u<N_2}e(\frac{k!z_1}{(k-A-2)!}\prod\limits_{i\leq A+1}h_iu_1^{k-A}+f_{A+1}({u_1,h_1,...,h_A,h_{A+1}}))\bigg|\bigg)^{2^{n-A-1}y}\\
G_3=&\bigg(\sum\limits_{h_{A+1}<\frac{K_{A+1}}{(k-A)N^{k-A-1}},h_A\in\text{Bad}(A)}\bigg|\sum\limits_{N_1<u<N_2}e(\frac{k!z_1}{(k-A-1)!}\prod\limits_{i\leq A+1}h_iu_1^{k-A}+f_{A+1}({u_1,h_1,...,h_A,h_{A+1}}))\bigg|\bigg)^{2^{n-A-1}y}
\end{align*}
By the norm inequality we know
\begin{align*}
&\bigg(\sum\limits_{h_{A+1}>\frac{K_{A+1}}{(k-A)N^{k-A-1}}}\bigg|\sum\limits_{N_1<u<N_2}e(\frac{k!z_1}{(k-A-1)!}\prod\limits_{i\leq A+1}h_iu_1^{k-A}+f_{A+1}({u_1,h_1,...,h_A,h_{A+1}}))\bigg|\bigg)^{2^{n-A-1}y}\\
\leq&\sum\limits_{h_{A+1}>\frac{K_{A+1}}{(k-A)N^{k-A-1}}}\bigg|\sum\limits_{N_1<u<N_2}e(\frac{k!z_1}{(k-A-1)!}\prod\limits_{i\leq A+1}h_iu_1^{k-A}+f_{A+1}({u_1,h_1,...,h_A,h_{A+1}}))\bigg|^{2^{n-A-1}y}\\
&\bigg(2N^\theta-\frac{K_{A+1}}{(k-A)N^{k-A-1}}\bigg)^{2^{n-A-1}y-1}
\end{align*}
The same as $G_1$, we know
\begin{align*}
&\bigg(\sum\limits_{h_{A+1}<\frac{K_{A+1}}{(k-A)N^{k-A-1}},h_A\in\text{Bad}(A)}\bigg|\sum\limits_{N_1<u<N_2}e(\frac{k!z_1}{(k-A-1)!}\prod\limits_{i\leq A+1}h_iu_1^{k-A}+f_{A+1}({u_1,h_1,...,h_A,h_{A+1}}))\bigg|\bigg)^{2^{n-A-1}y}\\
\leq&\sum\limits_{h_{A+1}<\frac{K_{A+1}}{(k-A)N^{k-A-1}},h_A\in\text{Bad}(A)}\bigg|\sum\limits_{N_1<u<N_2}e(\frac{k!z_1}{(k-A-1)!}\prod\limits_{i\leq A+1}h_iu_1^{k-A}+f_{A+1}({u_1,h_1,...,h_A,h_{A+1}}))\bigg|^{2^{n-A-1}y}\\
&N^{1-2^{n-A-1}y}\bigg(2N^\theta-\frac{K_{A+1}}{(k-A)N^{k-A-1}}\bigg)^{2^{n-A-1}y-1}
\end{align*}
Here we find three cases here:\\
\emph{case 1:}
\begin{align*}
&|\sum\limits_{{N_1}<u<{N_2}}{e(\frac{k!z_1}{(k-A)!}\prod\limits_{i\leq A}h_iu^{k-A}+f_A({u_1,h_1,...,h_A})}|^{2^{n-A}y}\\
\ll&\sum\limits_{h_{A+1}<\frac{K_{A+1}}{(k-A)N^{k-A-1}}}|S^*(\frac{k!z_1}{(k-A-1)!}\prod\limits_{i\leq A+1}h_i;n-A-1,k-A-1)|\bigg(2N^\theta-\frac{K_{A+1}}{(k-A)N^{k-A-1}}\bigg)^{2^{n-A-1}y-1}
\end{align*}
\emph{case 2:}
\begin{align*}
|\sum\limits_{{N_1}<u<{N_2}}{e(\frac{k!z_1}{(k-A)!}\prod\limits_{i\leq A}h_iu^{k-A}+f_A({u_1,h_1,...,h_A})})|^{2^{n-A}y}
\ll\bigg(\frac{K_1}{kN^{k-A-1+\frac{\theta}{4(k-A+1)}}}\bigg)^{2^{n-A-1}y}
\end{align*}
\emph{case 3:}
\begin{align*}
&|\sum\limits_{{N_1}<u<{N_2}}{e(\frac{k!z_1}{(k-A)!}\prod\limits_{i\leq A}h_iu^{k-A}+f_A({u_1,h_1,...,h_A})})|^{2^{n-A}y}\\
\ll&\sum\limits_{h_{A+1}<\frac{K_{A+1}}{(k-A)N^{k-A-1}},h_A\in\text{Bad}(A)}|S^*(\frac{k!z_1}{(k-A-1)!}\prod\limits_{i\leq A+1}h_i;n-A-1,k-A-1)|\\
&N^{1-2^{n-A-1}y}\bigg(2N^\theta-\frac{K_{A+1}}{(k-A)N^{k-A-1}}\bigg)^{2^{n-A-1}y-1}
\end{align*}
What we will do next is to ensure the frequency of these two cases.
We called this step "Efficienct Partition of Weyl Difference". Base on this method, we partition the integral of exponential sum into 3 parts. The integral $I$ is which can kept differential inequality case 1 after $n$ steps; the integral $II$ is the one which only can kept the inequality until $j\geq A$, when $j>A$ it may turns to case II. When we discuss the case I and case II, they will produce the sum of case III naturally. We will calculate the contribution of that in Section 5.
\section{The contribution of case I and case II}
In sum I we just need to calculate case I. In this situation, we know
\[
|S^*(z_A;n-A,k-A)|\ll\sum\limits_{|h_1|\geq\frac{K_1}{kN^{k-1}}}|S^*(z_{A+1};n-A-1,k-A-1)|\bigg(2N^\theta-\frac{K_{A+1}}{(k-A)N^{k-A-1}}\bigg)^{2^{n-A-1}y-1}
\]
By using this step $n-A$ times, we know
\begin{align*}
|S^*(z_1;n,k)|\ll&\sum\limits_{|h_1|\geq\frac{K_1}{kN^{k-1}}}...\sum\limits_{|h_n|\geq\frac{K_n}{(k-n+1)N^{k-n}}}|S^*(\frac{k!z_1}{(k-n)!}\prod\limits_{i\leq n}h_i;0,k-n)|\prod\limits_{j\leq n}\bigg(2N^{\theta}-\frac{K_j}{(k-j+1)N^{k-j}}\bigg)^{2^{n-j}y-1}\\
=&\sum\limits_{|h_1|\geq\frac{K_1}{kN^{k-1}}}...\sum\limits_{|h_n|\geq\frac{K_n}{(k-n+1)N^{k-n}}}|\sum\limits_{{N_1}<u<{N_2}}{e(\frac{k!z_1}{(k-n)!}\prod\limits_{i\leq n}h_iu^{k-n}+f_n({u_1,h_1,...,h_n})}|^{y}\\
&\prod\limits_{j\leq n}\bigg(2N^{\theta}-\frac{K_j}{(k-j+1)N^{k-j}}\bigg)^{2^{n-j}y-1}
\end{align*}
Hence we have
\[
|S^*(z_1;n,k)|\ll N^{\theta y(1-\frac{1}{4(k+1)})+\epsilon}\prod\limits_{j\leq n}\bigg(2N^{\theta}-\frac{K_j}{(k-j+1)N^{k-j}}\bigg)^{2^{n-j}y}
\]
Here we use lemma 2.2. From the structure of the expression, we would better to set
\[
K_j=
(k-j+1)N^{k-j}(2N^\theta-N^{\delta_j})(\delta_j<\theta,j>A)
\]
Now we need to calculate the proportion of case I and case II. We firstly find the condition that case II happened in $j-th$ step. Ignoring some constants, for those $j\geq A$,
\[
\sum\limits_{h_j\geq\frac{K_j}{(k-j+1)N^{k-j}}}\bigg|\sum\limits_{u_j}e(\frac{k!z_1}{(k-j)!}\prod\limits_{i\leq n}h_iu^{k-j}+f_j)\bigg|^{2^{n-j}y}N^{\delta_j(2^{n-j}y-1)}\geq\frac{C_jK_j^{2^{n-j}y}}{N^{(k-j-1+\frac{1}{4(k-j+1)}\theta)2^{n-j}y}}
\]
What we need is
\[
\min\limits_{h_j}\bigg|\sum\limits_{u_j}e(\frac{kz_1}{k-j}\prod\limits_{i\leq n}h_iu^{k-j}+f_j)\bigg|^{2^{n-j}y}\geq \frac{C_jK_j^{2^{n-j}y-1}N^{(1-\frac{1}{k-j+1})2^{n-j}y\theta}}{N^{(k-j)(2^{n-j}y-1)+\delta_j(2^{n-j}y-1)}}\asymp N^{(\theta-\delta_j)(2^{n-j}y-1)+(1-\frac{1}{4(k-j+1)})2^{n-j}y\theta}
\]
It is equal to find the smallest $\delta_j>0$, making 
\[
\bigg|\sum\limits_{u_j}e(\frac{k!z_1}{(k-j)!}\prod\limits_{i\leq n}h_iu^{k-j}+f_j)\bigg|\geq C_jN^{\theta-\frac{\theta}{4(k-j+1)}-(\theta-\delta_j)(1-\frac{1}{2^{n-j}y})}
\]
Let $T_i(i\leq k)$ are some monomials, and we set
\[
S_j(\alpha)=\sum\limits_{N-N^\theta\leq u\leq N+N^\theta}e(T_k(\alpha)),j=1,2,...,j
\]
Then to find the measure of $\alpha$ that make the inequality hold true for first time, we need this lemma:
\begin{lemma}
If we have
\[
\int_{(0,1]}|\sum\limits_{N_1<u<N_2}e(\alpha u^k)|^m\mathrm{d}\alpha\leq N^{\theta(m-\phi(m))}
\]
Then we have the measure inequality
\[
\mathfrak{L}\Big(\alpha:|S_k(\alpha)|\leq N^{h+\theta/2}\Big)\leq N^{\theta(m-\phi(m))-hm}
\]
\end{lemma}
\begin{proof}
We know
\[
\int_{(0,1]}|S(\alpha)|^m\mathrm{d}\alpha\leq N^{\theta(m-\phi(m))}
\]
So we know
\[
x^m\mathfrak{L}\Big(\alpha:\frac{|S_k(\alpha)|}{N^{\theta/2}}\leq x\Big)\leq\int_{(0,1]}|S(\alpha)|^m\mathrm{d}\alpha\leq N^{\theta(m-\phi(m))}
\]
which shows
\[
\mathfrak{L}\Big(\alpha:\frac{|S_k(\alpha)|}{N^{\theta/2}}\leq x\Big)\leq N^{\theta(m-\phi(m))-hm}
\]
then we complete the proof.
\end{proof}
To find the exact $\alpha$, our calculation may hold a self-reference, and we will use this character in Section 6. To calculate the probability of our difference stop in step II, we need another lemma.
\begin{lemma}
Let $\alpha$ is a uniformly distributed random variable on $(0,1]$, $C_j$ are some positive integers, then we know
\[
S_j(C_j\alpha)=\sum\limits_{N_1<u<N_2}e(\alpha(C_ju^j+T_j))(j=k-n,...,k-A)
\]
are some independent random variable rely on $\alpha$.
\end{lemma}
\begin{proof}
In this situation, the variable is actually just rely on $\alpha$, which means we just need to prove
\[
\textbf{E}[S_i^{k_1}(C_i\alpha)S_j^{k_2}(C_j\alpha)]=\textbf{E}[S_i^{k_1}(C_i\alpha)]\textbf{E}[S_j^{k_2}(C_j\alpha)]
\]
and we know all of them are equal to $0$, so we complete the proof.
\end{proof}
Therefore we can calculate the measure of $\alpha$ of integral $II$ as follows, here we calculate the probability that the difference end in step $l$.
\[
\prod\limits_{A\leq j\leq l}P\Big(\alpha:\Big|\sum\limits_{u}e(C_{n-j}\alpha u^{n-j+f_j})\Big|\leq N^{h}\Big)P\Big(\alpha:\Big|\sum\limits_{u}e(C_{n-l-1}\alpha u^{n-{l}-1}+f_j)\Big|>N^{h}\Big)
\asymp N^{\theta(m-\phi(m))-hm}
\]
where
\begin{align*}
&h=\theta-\frac{\theta}{4(k-j+1)}-(\theta-\delta_j)(1-\frac{1}{2^{n-j}y})\\
&C_k=\frac{k!z_1}{(k-j)!}\prod\limits_{i\leq n}h_i
\end{align*}
Then we can choose $\delta_j=\lambda_j\theta$. To satisfy the condition in lemma 3.2, we need $\delta_j-\theta=o(1)$, which means
\[
\frac{\theta}{4(k-j+1)}+(\theta-\delta_j)(1-\frac{1}{2^{n-j}y})=o(1)
\]
which means
\[
\lambda_j=1-\frac{m}{(k-j+1)(m-2^j)}=\frac{4k-4j+3}{4(k-j+1)}(1+C\frac{2^j}{m})+o(1)
\]
and $|C|<1$. Therefore it is safe to choose
\[
\lambda_j=\frac{4k-4j+3}{4(k-j+1)}
\]
and we set
\[
\lambda=\max\limits_{j}\lambda_j=\frac{4k-4A+3}{4(k-A+1)}
\]
In this condition, we know
\[
p_l\ll N^{\theta(2^l\lambda_l-\phi(m))}=o(N^{\theta(\lambda2^l-\phi(m))})
\]
To calculate the exact bound of sum $I$, we have
\[
2N^{\theta}-\frac{K_j}{(k-j+1)N^{k-j}}\ll N^{\lambda\theta}
\]
On the other hand, we have
\[
\frac{3}{2}N^{\theta}-\frac{K_j(k-j)}{kN^{k-j-\theta}}=\frac{3}{2}N^{\theta}+O(N^{\theta-k-j})
\]
Now for the sum $I$, which is only consisted by case 1, we have
\[
\int_{z:I}|\sum\limits_{{N_1}<u<{N_2}}{e(z_1u^k)}|^{U}\mathrm{d}z\ll N^{((2^n-2^{n-A+1})y+\lambda y(2^{n-A+1}-1)+\frac{(4k+3)y}{4(k+1)})\theta+\epsilon}
\]
where $\frac{m}{2^n}-1<y<\frac{m}{2^n}$ and we have
\[
\int_{z:I}|\sum\limits_{{N_1}<u<{N_2}}{e(z_1u^k)}|^{m}\mathrm{d}z\ll N^{\theta(1+y^{-1})((2^n-(1-\lambda)2^{n-A+1})y+(\frac{4k+3}{4(k+1)}-\lambda)y)+\epsilon}
\]
For the sum $II$, we calculate it as we do in sum $I$. With the probability $p_l$, the difference process terminates at step $l$ and falls into case 2. We have
\begin{align*}
|S^*(z_1;n,k)|\ll&\sum\limits_{|h_1|\geq\frac{K_1}{kN^{k-1}}}...\sum\limits_{|h_l|\geq\frac{K_l}{(k-l+1)N^{k-l}}}|S^*(\frac{k!z_1}{(k-l)!}\prod\limits_{i\leq l}h_i;n-l,k-l)|\\
&\prod\limits_{j\leq A}N^{(2^{n-j}y-1)\theta}\prod\limits_{A+1\leq j\leq l}N^{(2^{n-j}y-1)\lambda\theta}\\
=&\sum\limits_{|h_1|\geq\frac{K_1}{kN^{k-1}}}...\sum\limits_{|h_l|\geq\frac{K_l}{(k-l+1)N^{k-l}}}|\sum\limits_{{N_1}<u<{N_2}}{e(\frac{k!z_1}{(k-l)!}\prod\limits_{i\leq l}h_iu^{k-l}+f_l)}|^{2^{n-l}y}\\
&\prod\limits_{j\leq A}N^{(2^{n-j}y-1)\theta}\prod\limits_{A+1\leq j\leq l}N^{(2^{n-j}y-1)\lambda\theta}\\
\ll&\bigg(\frac{K_l}{kN^{k-l-1+\frac{\theta}{4(k-l+1)}}}\bigg)^{2^{n-l+1}y}\prod\limits_{j\leq A}N^{2^{n-j}y\theta}\prod\limits_{A+1\leq j\leq l}N^{2^{n-j}y\lambda\theta}\\
\ll&N^{\theta(m-(1-\lambda)2^{n-A}y-2^{n-l}y+\frac{2^{n-l-2}y}{k-l+1})}
\end{align*}
where $l\geq A$. Here we use the condition that our differential cannot be kept. For the integral, We have 
\begin{align*}
&\int_{z:II}|\sum\limits_{{N_1}<u<{N_2}}{e(zu^k)}|^{m}\mathrm{d}z\\
\ll&\sum\limits_{A\leq l\leq n}p_lN^{\theta(m-(1-\lambda)2^{n-A}y-2^{n-l}y+\frac{2^{n-l-2}y}{k-l+1})(1+\frac{1}{y})}\\
\ll&\sum\limits_{A\leq l\leq n}N^{\theta(m-(1-\lambda)2^{n-A}y-2^{n-l}y+\frac{2^{n-l-2}y}{k-l+1})(1+\frac{1}{y})+\theta(2^l\lambda_l-\phi(m))}\\
\ll_n&\max\limits_lN^{\theta(m-(1-\lambda)2^{n-A}y-2^{n-l}y+\frac{2^{n-l-2}y}{k-l+1})(1+\frac{1}{y})+\theta(2^l\lambda-\phi(m))}
\end{align*}
where $\epsilon$ is any given positive number.
\section{The contribution of case III}
Now we put our sight to the case III, the bad $C\alpha$. We know there are two way to produce the bad $C\alpha$: the difference of case I(the last step) and case II; the difference of case III. We know when we do the differencing step, case III will contribute to the case II, but the order of the count of case III is less than case II naturally, so that when it produce some $C\alpha$ not bad, and the contribution can also be ignored. Therefore we just need to calculate the sums in case III which is never turned to case I or case II. 
According to the discussion, we know
\begin{align*}
&\sum\limits_{h_{l+1}<\frac{K_{l+1}}{(k-l)N^{k-l-2}},h_l\in\text{Bad}(A)}|S^*(\frac{k!z_1}{(k-l-1)!}\prod\limits_{i\leq l+1}h_i;n-l-1,k-l-1)|\\
\ll&\sum\limits_{\substack{\forall j\geq l:h_{j+1}<\frac{K_{j+1}}{(k-j)N^{k-j-1}}\\h_n\in\text{Bad}(n)}}|S^*(\frac{k!z_1}{(k-n)!}
\prod\limits_{i\leq n}h_i;0,k-n)|\\
&\prod\limits_{l\leq j\leq n}N^{1-2^{n-j-1}y}\bigg(2N^{\theta}-\frac{K_j}{(k-j+1)N^{k-j}}\bigg)^{2^{n-j}y}+o\bigg(\bigg(\frac{K_l}{kN^{k-l-1+\frac{\theta}{4(k-l+1)}}}\bigg)^{2^{n-l+1}y}\bigg)\\
\ll&N^{\theta y}\prod\limits_{l\leq j\leq n}N^{1-2^{n-j-1}y}\bigg(2N^{\theta}-\frac{K_j}{(k-j+1)N^{k-j}}\bigg)^{2^{n-j}y}+o\bigg(\bigg(\frac{K_l}{kN^{k-l-1+\frac{\theta}{4(k-l+1)}}}\bigg)^{2^{n-l+1}y}\bigg)
\end{align*}
Hence the contribution is rely on the first expression. Our main goal is to prove
\[
N^{\theta y}\prod\limits_{l\leq j\leq n}N^{1-2^{n-j-1}y}\bigg(2N^{\theta}-\frac{K_j}{(k-j+1)N^{k-j}}\bigg)^{2^{n-j}y}=o\bigg(\bigg(\frac{K_l}{kN^{k-l-1+\frac{\theta}{4(k-l+1)}}}\bigg)^{2^{n-l+1}y}\bigg)
\]
when $l>A$, we have
\[
2N^{\theta}-\frac{K_j}{(k-j+1)N^{k-j}\theta}\sim N^\frac{k-j}{k-j+1}\leq N^{\lambda\theta}
\]
so we know
\begin{align*}
&\max\limits_{l}\prod\limits_{l\leq j\leq n}N^{1-2^{n-j}y}\bigg(2N^{\theta}-\frac{K_j}{(k-j+1)N^{k-j}}\bigg)^{2^{n-j-1}y}\bigg(\frac{K_l}{kN^{k-l-1+\frac{\theta}{4(k-l+1)}}}\bigg)^{-2^{n-l+1}y}\\
\leq&\exp(\log N(y\theta+\max\limits_l(2^{n-l+1}-1)(\lambda-3y/2)+n-l))=o(1)
\end{align*}
Calculating the exponent left, we know the main term of case III in $l-th$ step is
\[
\sum\limits_{h_{l+1}<\frac{K_{l+1}}{(k-l)N^{k-l-1}},h_l\in\text{Bad}(A)}|S^*(\frac{k!z_1}{(k-l-1)!}\prod\limits_{i\leq l+1}h_i;n-l-1,k-l-1)|=o\bigg(\bigg(\frac{K_l}{kN^{k-l-1+\frac{\theta}{4(k-l+1)}}}\bigg)^{2^{n-l-1}y}\bigg)
\]
Hence we know the contribution of case III can be ignored.
\section{The final conclusion}
Putting these integrals together, we have
\[
\int_{\mathfrak{m}}\bigg|\sum\limits_{{N_1}<u<{N_2}}{e(zu^k)}\bigg|^{m}\mathrm{d}z\ll N^{\beta_1}+N^{\beta_2}
\]
where
\begin{align*}
\beta_1=&\theta(y+1)(2^n-(1-\lambda)2^{n-A+1}+\frac{4k+3}{4(k+1)}-\lambda)\\
\beta_2=&\theta(m-(1-\lambda)2^{n-A}y-2^{n-l}y+\frac{2^{n-l-2}y}{k-l+1})(1+\frac{1}{y})+\theta(2^l\lambda-\phi(m))
\end{align*}
and we have the lemma follows:
\begin{lemma}
Setting
\[
G^*=\inf\limits_{2^n(y+1)>m>2^ny}\inf\limits_{A<n,\phi(m)\leq m-G^*}\max\limits_{l}\max\left\{\beta_1,\beta_2\right\}
\]
when there exist an $\epsilon>0$, such that $k<m^{1/2-\epsilon}$, we have
\[
\frac{G^*}{m}=1-\frac{1}{2k+2}+o(1)
\]
\end{lemma}
\begin{proof}
First, we prove that we have $G^*=\phi(m)$. There are two cases:\\
\emph{case 1}
\[
G^*=\inf\limits_{2^n(y+1)>m>2^ny}\inf\limits_{A<n,\phi(m)\leq m-G^*}\max\limits_{l}\beta_1
\]
\emph{case 2}
\[
G^*=\inf\limits_{2^n(y+1)>m>2^ny}\inf\limits_{A<n,\phi(m)\leq m-G^*}\max\limits_{l}\beta_2
\]
In \emph{case 1} we know whatever $\phi(m)$ is, it will not influence the final result. In \emph{case 2}, we set $\phi(m)=\theta m-G^*-\epsilon$, where $\epsilon$ is a given number; and we set $\beta_2=\beta_2(A,l;y,n;\phi)$, then we have
\[
G^*=\inf\limits_{2^n(y+1)>m>2^ny}\inf\limits_{A<n,\phi(m)\leq m-G^*}\max\limits_{l}\beta_2(A,l;y,n;\theta m-G^*-\epsilon)
\]
However, it is easy to find $G^*$ is decreasing with $\phi(m)$, so we have
\begin{align*}
&\inf\limits_{2^n(y+1)>m>2^ny}\inf\limits_{A<n,\phi(m)\leq m-G^*}\max\limits_{l}\beta_2(A,l;y,n;\theta m-G^*-\epsilon)\\
\geq&\inf\limits_{2^n(y+1)>m>2^ny}\inf\limits_{A<n,\phi(m)\leq m-G^*}\max\limits_{l}\beta_2(A,l;y,n;\theta m-G^*-\epsilon/2)
\end{align*}
This is in contradiction with the definition of the infimum, so we can set $\phi(m)=G^*$. In \emph{case 1}, we know
\[
\beta_1=\theta(y+1)(2^n-(1-\lambda)2^{n-A+1}+\frac{4k+3}{4(k+1)}-\lambda)
\]
Substitute $\lambda$ into this expression and we have
\begin{align*}
G^*=\theta(y+1)(2^n-\frac{2^{n-A-1}}{k-A+1}+\frac{1}{4(k-A+1)}-\frac{1}{4(k+1)})
\end{align*}
If we want to make our calculation nontrivial, we need to set $2^A=o(y)$ and $2^Ay=o(m)$. In this condition, we have
\[
G^*=\inf\limits_{n,y,A}\theta(1+\frac{1}{y})((1-\frac{1}{2^{A+1}(k+1)})m+O(Ayk^{-2})+O(\frac{m}{2^Ak^2}))
\]
which shows
\[
G^*=\inf\limits_{n,y,A}(m+\frac{m}{y}-\frac{m}{2^{A+1}(k+1)})\theta+O(Ayk^{-2})+O(\frac{m}{2^Ak^2})
\]
We choose $A,y,n$ as follows:
\begin{align*}
A_0=&0\\
y_0=&\lfloor{m^{\frac{1}{2}-c_1{\delta}}}\rfloor+1\\
n_0=&\lfloor\frac{\log{m}}{\log{2}}(\frac{1}{2}+c_1{\delta})\rfloor
\end{align*}
The power of $N$ in $G^*$ can be write as
\[
G^*_0=\theta(m-\frac{m}{2k+2})+m^{1/2+c_1\delta}
\]
so we need $-1<c_2<c_1<1$, and Choosing $\delta=1/2-\epsilon$ is optimal in sense of exact power of $N$. Therefore we find
\[
\frac{G^*}{m}=\inf\limits_{2^n(y+1)>m>2^ny}\inf\limits_{A<n,\phi(m)\leq m-G^*}\max\limits_{l}\beta_1=1-\frac{1}{2k+2}
\]
In this condition, for any $\delta>0$ we know
\begin{align*}
\beta_2=&\theta(m-(1-\lambda)2^{n-A+1}y-2^{n-l}y+\frac{2^{n-l-2}y}{k-l+1})(1+\frac{1}{y})+\theta(2^l\lambda-\phi(m))\\
\leq&\theta(m-\frac{m}{2k+2}+y)(1+\frac{1}{y})+\theta(\frac{m}{4y(k-A+1)}-\frac{m}{2k})\\
\leq&\theta(m-\frac{m}{2k+2})+m^{1/2+c_1\delta}
\end{align*}
setting $\beta_1=\beta_1(A,l;y,n)$, hence we know there is $(A_0,l_0;y_0,n_0)=(A,l;y,n)$, making
\[
\max\left\{\beta_1(A,l;y,n),\beta_2(A,l;y,n;\theta m-G^*)\right\}=G^*_0
\]
It show
\[
\frac{G^*_0}{m}=\frac{G^*}{m}+o(1)
\]
and we complete the proof.
\end{proof}
Hence we have
\[
\int_{\mathfrak{m}}\bigg|\sum\limits_{{N_1}<u<{N_2}}{e(zu^k)}\bigg|^{m}\mathrm{d}z\ll N^{\theta m(\frac{2k+1}{2k+2}+o(1))}
\]
where $\epsilon$ is any given positive number, and $C>0$ is a constants not rely on $m$.
\begin{remark}
We need to point out that our estimate improve the theorem 1.2, the generation of Vinogradov's main conjecture. In some way, we can see it as a better version of the main conjecture in incomplete arc. For an example, from the pointwise estimate of Vinogradov conjecture,
\[
\int_{\mathfrak{m}}\bigg|\sum\limits_{{N_1}<u<{N_2}}{e(zu^k)}\bigg|^{m}\mathrm{d}z\ll N^{\theta m(1-\frac{\gamma}{k^2\log k})}
\]
where $\gamma=\frac{19}{1600}$(By the method in \cite{PP1986}) so we can find our estimate reduces the exponent for almost $\theta m/2(k+1)$ The reason why we can improve the pointwise estimate is that we use various estimate for various $\alpha$, and combined various tools. Moreover, another important reason is that we have ignored some bad $\alpha$ when we set the problem, so we can overtake the pointwise bound.
\end{remark}
\begin{remark}
From the process of our prove, we know
\[
\int_{\mathfrak{m^*}}\bigg|\sum\limits_{{N_1}<u<{N_2}}{e(zu^k)}\bigg|^{m}\mathrm{d}z\leq|\mathfrak{m^*}|N^{\theta m(\frac{2k+1}{2k+2}+m^{-1/2-\delta})}
\]
This conclusion is true for any $\mathfrak{m}$ whose measure is positive, even the measure of $\mathfrak{m}$ is rely on $N,\theta$
\end{remark}
\section{Appendix}
[\emph{The proof of lemma 2.2}]From the definition of $T(K)$, we know
\begin{align*}
T(K)=&\sum\limits_{N-N^{\theta}<x\leq N+N^{\theta}}\sum\limits_{(f(x)+K)^{1/k}\geq y\geq(f(x)-K)^{1/k}}1\\
=&\sum\limits_{N-N^{\theta}<x\leq N+N^{\theta}}\lfloor(f(x)+K)^{1/k}\rfloor-\lfloor(f(x)-K)^{1/k}\rfloor\\
=&\sum\limits_{N-N^{\theta}<x\leq N+N^{\theta}}((f(x)+K)^{1/k}-(f(x)-K)^{1/k})\\
-&\sum\limits_{N-N^{\theta}<x\leq N+N^{\theta}}\big(\left\{(f(x)+K)^{1/k}\right\}-\left\{(f(x)-K)^{1/k}\right\}\big)
\end{align*}
We know
\[
\lfloor(f(x)+K)^{1/k}\rfloor-\lfloor(f(x)-K)^{1/k}\rfloor=(\lfloor(x^k+K)^{1/k}\rfloor-\lfloor(x^k-K)^{1/k}\rfloor)(1+O(\frac{1}{x}))
\]
Setting $N_1=N-N^\theta,N_2=N+N^\theta$, It's obvious that
\begin{align*}
&\sum\limits_{N_1<n\leq N_2}\bigg|\int_n^{n+1}(x^k-K)^{1/k}\mathrm{d}x-((n+1)^k-K)^{1/k}\bigg|\\
=&\sum\limits_{N_1<n\leq N_2}(1-\frac{K}{n^k})^{1/k-1}(1+O(\frac{1}{N}))=2N^{\theta}+O(N^{\theta-1})
\end{align*}
and
\[
\sum\limits_{N_1<n\leq N_2}\bigg|\int_n^{n+1}(x^k+K)^{1/k}\mathrm{d}x-((n+1)^k+K)^{1/k}\bigg|=2N^{\theta}+O(N^{\theta-1})
\]
By Pascal theorem and the dominated convergence theorem, we can exchange the order of integral and sum as follows:
\begin{align*}
\int_{N_1}^{N_2}(x^k-K)^{1/k}\mathrm{d}x=&K^{2/k}\int_{N_1K^{-1/k}}^{N_2K^{-1/k}}(x^k-1)^{1/k}\mathrm{d}x\\
=&K^{2/k}\int_{N_1K^{-1/k}}^{N_2K^{-1/k}}\sum\limits_{n=0}^\infty\binom{1/k}{n}(-1)^nx^{1-kn}\mathrm{d}x\\
=&K^{2/k}\sum\limits_{n=0}^\infty\int_{N_1K^{-1/k}}^{N_2K^{-1/k}}\binom{1/k}{n}(-1)^nx^{1-kn}\mathrm{d}x
\end{align*}
Calculating the integral inside directly, we have
\begin{align*}
\int_{N_1}^{N_2}(x^k-K)^{1/k}\mathrm{d}x=&\sum\limits_{n=0}^\infty\binom{1/k}{n}(-1)^n K^n\frac{(N+N^{\theta})^{2-kn}-(N-N^{\theta})^{2-kn}}{2-kn}\\
=&2N^{1+\theta}-\frac{K}{kN^{k-1-\theta}}+O(\frac{K^{2}}{N^{2k-2-\theta}})
\end{align*}
and the same as what we do above, we have
\[
\int_{N_1}^{N_2}(x^k+K)^{1/k}\mathrm{d}x=2N^{1+\theta}+\frac{K}{kN^{k-1-\theta}}+O(\frac{K^{2}}{N^{2k-2-\theta}})
\]
So we can get the main term:
\[
I_1=\sum\limits_{N-N^{\theta}<x\leq N+N^{\theta}}((x^k+K)^{1/k}-(x^k-K)^{1/k})=\frac{2K}{kN^{k-1-\theta}}+O(\frac{K^{2}}{N^{2k-2-\theta}})+O(N^{\theta-1})
\]
For the second sum, we know
\[
\sum\limits_{N-N^{\theta}<x\leq N+N^{\theta}}\big(\left\{(x^k+K)^{1/k}\right\}-\left\{(x^k-K)^{1/k}\right\}\big)=O(N^\theta)
\]
Thus we know
\[
T(K)=\frac{2K}{kN^{k-1-\theta}}(1+O(\frac{1}{N}))+O(\frac{K^{2}}{N^{2k-2-\theta}})+O(N^\theta)
\]
then we prove lemma 3.1.
\bibliographystyle{plain}
\nocite{*}
\bibliography{References}
\end{document}